\documentclass[a4paper,12pt,english,reqno]{amsart}
\usepackage{amsmath, amsthm, amssymb}

\usepackage{geometry}
\usepackage{fancyhdr}

\usepackage{xcolor}
\usepackage[all]{xy}
\usepackage{comment}
\definecolor{violet}{rgb}{0.0,0.2,0.7}
\definecolor{rouge2}{rgb}{0.8,0.0,0.2}
\usepackage{hyperref}
\usepackage{mathrsfs}
\hypersetup{
    bookmarks=true,         
    unicode=false,          
    pdftoolbar=true,        
    pdfmenubar=true,        
    pdffitwindow=false,     
    pdfstartview={FitH},    
    pdftitle={},    
    pdfauthor={},     
    colorlinks=true,       
   linkcolor=rouge2,          
    citecolor=violet,        
    filecolor=black,      
    urlcolor=cyan}           

\newcommand{\R}{\mathbb{R}}
\newcommand{\C}{\mathbb{C}}

\newcommand{\Z}{\mathbb{Z}}

\renewcommand{\d}{\partial}
\newcommand{\db}{\overline\partial}
\newcommand{\ddb}{\partial\overline\partial}

\newcommand{\ep}{\varepsilon}

\renewcommand{\div}{\mathrm{div}}
\newcommand{\Om}{\Omega}
\newcommand{\om}{\omega}
\newcommand{\omte}{\omega_{t,\ep}}
\newcommand{\ome}{\omega_{\ep}}

\newcommand{\Ric}{\mathrm{Ric} \,}

\newcommand{\ud}{|\sigma|^2+1}
\newcommand{\rr}{^{\otimes p}}
\newcommand{\tr}{\mathrm{tr}}

\newcommand{\la}{\langle}
\newcommand{\ra}{\rangle}

\renewcommand{\ge}{\geqslant}

\renewcommand{\le}{\leqslant}
\newcommand{\vp}{\varphi}

\theoremstyle{plain}
\newtheorem{thm}{Theorem}[section]
\newtheorem{prop}[thm]{Proposition}
\newtheorem{propdef}[thm]{Proposition-Definition}
\newtheorem{coro}[thm]{Corollary}
\newtheorem{lemm}[thm]{Lemma}

\newtheorem{bigthm}{Theorem}

\theoremstyle{definition}
\newtheorem{defi}[thm]{Definition}
\newtheorem{rema}[thm]{Remark}
\newtheorem{set}[thm]{Setting}

\numberwithin{equation}{section}

\begin{document}

\title[A Bochner principle and its applications for manifolds in class $\mathcal C$]{A Bochner
principle and its applications to Fujiki class $\mathcal C$ manifolds
with vanishing first Chern class}

\author[I. Biswas]{Indranil Biswas}

\address{School of Mathematics, Tata Institute of Fundamental
Research, Homi Bhabha Road, Mumbai 400005, India}

\email{indranil@math.tifr.res.in}

\author[S. Dumitrescu]{Sorin Dumitrescu}

\address{Universit\'e C\^ote d'Azur, CNRS, LJAD, France}

\email{dumitres@unice.fr}

\author[H. Guenancia]{Henri Guenancia}

\address{Institut de Math\'ematiques de Toulouse; UMR 5219, Universit\'e de Toulouse; CNRS, UPS,
118 route de Narbonne, F-31062 Toulouse Cedex 9, France}

\email{henri.guenancia@math.cnrs.fr}



\date{}

\begin{abstract}
We prove a Bochner type vanishing theorem for compact complex manifolds $Y$ in Fujiki class
$\mathcal C$, with 
vanishing first Chern class, that admit a cohomology class $\lbrack \alpha \rbrack\,\in\,
H^{1,1}(Y,\,\R)$ which is numerically effective (nef) and has positive self-intersection (meaning 
$\int_Y \alpha^n \,>\, 0$, where $n\,=\,\dim_{\mathbb C} Y$). Using it, we prove that all holomorphic
geometric structures of affine type on such a manifold $Y$ are locally homogeneous on a non-empty
Zariski open subset. Consequently, if the geometric structure is rigid in the sense of Gromov, then the 
fundamental group of $Y$ must be infinite. In the particular case where the geometric 
structure is a holomorphic Riemannian metric, we show that the manifold $Y$ admits a finite 
unramified cover by a complex torus with the property that the pulled
back holomorphic Riemannian metric on the torus is translation invariant.
\end{abstract}

\maketitle

\tableofcontents

\section {Introduction} 

Yau's celebrated theorem \cite{Ya}, proving Calabi's conjecture, endows any compact 
K\"ahler manifold $X$ with vanishing real first Chern class (meaning $c_1(TX)\,=\,0$ in 
$H^2(X, \,\R)$) with a Ricci flat K\"ahler metric. Such metrics constitute an extremely 
useful tool for studying the geometry of these manifolds, known as Calabi-Yau manifolds. 
For example, by the well-known Bochner principle, any holomorphic tensor on $X$ must be 
parallel with respect to any Ricci flat K\"ahler metric \cite{Be}. The study of the holonomy 
of such a Ricci flat K\"ahler metric furnishes, in particular, an elegant proof of the 
Beauville--Bogomolov decomposition theorem that asserts that, up to a finite unramified 
cover, a Calabi-Yau manifold $X$ is biholomorphic to the product of a complex torus 
with a compact complex simply connected manifold with trivial first Chern class \cite{Be, 
Bo}.

In the special case where the second real Chern class of the Calabi-Yau manifold $X$ also 
vanishes, any Ricci flat K\"ahler metric on $X$ has vanishing sectional curvature. In that 
case, as a consequence of Bieberbach's theorem, $X$ actually admits a finite unramified 
cover which is a complex torus \cite{Be}. Notice that a compact K\"ahler manifold bearing 
a holomorphic affine connection on its holomorphic tangent bundle has vanishing real Chern 
classes \cite[p.~192--193, Theorem 4]{At}, and hence it admits an unramified cover
by a compact complex torus \cite{IKO}.

Using the Bochner principle, it was proved in \cite{Du} that holomorphic geometric 
structures of affine type (their definition is recalled in the paragraph following 
Definition \ref{hgs}) on any compact K\"ahler manifold $X$ with vanishing first Chern 
class are in fact locally homogeneous. Consequently, if the geometric structure satisfies 
the condition that it is rigid in the sense of Gromov, \cite{DG,Gr}, then $X$ admits a 
finite unramified cover which is a complex torus.

The aim in this paper is to generalize the above mentioned results to the broader context 
of Fujiki class $\mathcal C$ manifolds \cite{Fu, Fu2}. Recall that a compact complex 
manifold $Y$ is in Fujiki class $\mathcal C$ if it is the image of a holomorphic map from a
compact K\"ahler manifold. By an important result of Varouchas \cite{Va}, this is equivalent to the 
assertion that $Y$ admits a surjective holomorphic map from a compact K\"ahler manifold 
such that the map is a bimeromorphism.

Let $Y$ be a compact connected complex manifold, of complex dimension $n$, with trivial first
(real) Chern class. We assume that $Y$
\begin{itemize}
\item lies in Fujiki class $\mathcal C$, and

\item{} admits a numerically effective (nef) cohomology class 
$\lbrack \alpha \rbrack \,\in\, H^{1,1}(Y,\, \R)$ that has 
positive self-intersection, meaning $\int_Y \alpha^n \,>\, 0$ with $\alpha$ being a real closed 
smooth $(1,1)$--form representing the cohomology class $\lbrack \alpha \rbrack$.
\end{itemize}
By a result of Demailly and P\u{a}un in 
\cite{DP}, the combination of the two conditions above is equivalent to the condition that $Y$ admits a cohomology class $ \lbrack 
\alpha \rbrack \,\in\, H^{1,1}_{\partial \overline\partial}(Y,\,\R)$ which is both nef and big (see Section~\ref{positivity}, e.g. Remark~\ref{fujiki} and Corollary~\ref{cor1}). 

Under the above assumptions we prove the following Bochner type theorem for holomorphic tensors on $Y$ 
(see Corollary~\ref{determination}).

\begin{bigthm}
\label{thmA}
Let $Y$ be a complex compact manifold with trivial first Chern class. Assume that $Y$ admits a nef and big cohomology class $[\alpha]\in H^{1,1}_{\partial \overline\partial}(Y,\,\R)$. 

\noindent 
Then, there exists a closed, positive $(1,1)$-current $\om \in [\alpha]$ which induces a genuine Ricci-flat K\"ahler metric on a non-empty Zariski open subset $\Omega \subset Y$. 

\noindent
Furthermore, given any global holomorphic tensor $\tau \in H^0(Y,\,{T_Y}^{\otimes p}\otimes {T_Y^*}^{\otimes q})$ with $p,q\, \geq \, 0$, the restriction $\tau|_{\Om}$ is parallel with respect to $\om|_{\Om}$.
\end{bigthm}

The existence of the singular Ricci-flat metric $\om$ is essentially due to \cite{BEGZ}; see Section~\ref{RFK}. 
The statement about parallelism is new, although the techniques involved in its proof are mostly borrowed from 
\cite{Gue} (see also \cite{CP1}). Using this Bochner principle, we deduce the following result
(see Theorem~\ref{ouvert dense}).

\begin{bigthm}
\label{thmB}
Under the assumptions of Theorem~\ref{thmA}, the following holds:
\begin{enumerate}
\item There exists Zariski open subset $\emptyset \neq \Omega\subset Y$ such that any holomorphic geometric structure of affine type on $Y$ is locally homogeneous on $\Om$.

\item If $Y$ admits a rigid holomorphic geometric structure of affine type, then the fundamental group of $Y$ is infinite.
\end{enumerate} 
\end{bigthm}

For the particular case of a holomorphic Riemannian 
metric, we deduce using the above mentioned Bochner principle the following result (see Theorem~\ref{holo metric}).
 
\begin{bigthm}
\label{thmC}
Under the assumptions of Theorem~\ref{thmA}, assume additionally that $Y$ admits a holomorphic
Riemannian metric $g$. 

\noindent
Then, there is a finite unramified cover $\gamma\, :\, {\mathbb T}\, 
\longrightarrow\, Y$, where $\mathbb T$ is a complex torus, such that the pulled back holomorphic Riemannian 
metric $\gamma^*g$ on the torus ${\mathbb T}$ is translation invariant.
\end{bigthm}

The Zariski open subset $\Omega$ involved in Theorem~\ref{thmB} can be chosen to contain the smooth locus of the singular Ricci-flat metric $\om$ from Theorem~\ref{thmA}. We would actually conjecture that we could choose $\Omega=Y$. This was proved to be true in \cite{BD} for Moishezon manifolds (meaning manifolds that are bimeromorphic to some 
complex projective manifold \cite{Mo}) using deep positivity properties proved recently 
in \cite{CP2} for tensor powers of the cotangent bundle of \textit{projective} manifolds with pseudoeffective canonical class (see also the expository article \cite{Cl}). 
However it is still not known whether this holds in the set-up of K\"ahler manifolds. It should also 
be mentioned that it was proved in \cite{BD} that a compact simply connected manifold in 
Fujiki class $\mathcal C$ does not admit any holomorphic affine connection on its 
holomorphic tangent bundle. A compact complex manifold in Fujiki class 
$\mathcal C$ bearing a holomorphic affine connection has in fact trivial real Chern classes (its 
proof is identical to the proof in the K\"ahler case \cite{At}). We would conjecture that 
a compact complex manifold in Fujiki class $\mathcal C$ with trivial Chern classes (in the 
real cohomology) admits a finite unramified  cover which is  a complex torus.

\section{Fujiki class $\mathcal C$ manifolds and Bochner principle}

\subsection{Positivity property of $(1,1)$-classes.}
\label{positivity}
Let $X$ be a compact complex manifold of dimension $n$. We set $d^c\,:=
\,\frac{\sqrt{-1}}{2\pi}(\db-\d)$ the real operator so that $dd^c\,=\,\frac{\sqrt{-1}}{\pi}\d\db$. We shall denote by
$H^{1,1}_{\partial \overline\partial}(X)$ the $\ddb$--cohomology, in other words,
$$H^{1,1}_{\partial \overline\partial}(X)\,=\,
\{\alpha \,\in \,\mathcal C^{\infty}(X, \Omega_{X,\mathbb C}^{1,1})\,\mid\, d\alpha\,=\,0\} /
\{\ddb u\,\mid\, u\in \mathcal C^{\infty}(X,\C)\}\, .$$
Moreover, set
$$H^{1,1}_{\partial \overline\partial}(X,\, {\mathbb R})\,:=\, H^{1,1}_{\partial \overline\partial}(X)
\cap H^2(X,\, {\mathbb R}).$$

\begin{defi}
Let $X$ be a compact complex manifold of dimension $n$, and let $[\alpha]\,\in \,
H^{1,1}_{\partial \overline\partial}(X,\, {\mathbb R})$ be a cohomology class represented by a
smooth, closed $(1,1)$-form $\alpha$. We recall the standard terminology:
\begin{enumerate}
\item $[\alpha]$ is \textit{nef} if for any $\ep>0$, there exists a smooth representative
$\om_{\ep}\,\in\, [\alpha]$ such that $\omega_{\ep}\,\ge\, -\ep \om_X$, where $\om_X$ is some
fixed (independent of $\ep$) hermitian metric on $X$.

\item $[\alpha]$ has \textit{positive self-intersection} if $\int_X \alpha^n\,>\,0$. 

\item $[\alpha]$ is \textit{big} if there exists a K\"ahler current $T\,\in\, [\alpha]$, i.e., if
there exists a closed $(1,1)$-current $T\,=\,\alpha+dd^c u\,\in \,[\alpha]$ such that $T\,\ge\, \om_X$ in the sense of
currents, where $u\in L^1(X)$ and $\om_X$ is some hermitian metric on $X$.
\end{enumerate}
\end{defi}

If $X$ is a compact K\"ahler manifold or more generally a Fujiki manifold (see Definition~\ref{fuj}), then the $\partial \overline\partial$--cohomology 
coincides with the usual $\overline\partial$--cohomology (also called Dolbeault cohomology). 
Moreover, Demailly and P\u{a}un proved the following fundamental theorem which is the key 
result towards their numerical characterization of the K\"ahler cone:

\begin{thm}[{\cite[p.~1259, Theorem 2.12]{DP}}]\label{DP}
Let $X$ be a compact K\"ahler manifold, and let $[\alpha]\,\in\, H^{1,1}(X,\,\R)$ be a nef
class with positive self-intersection. Then, $[\alpha]$ is a big class. 
\end{thm}

Its converse is true as well, meaning all nef and big classes on a 
compact K\"ahler manifold have positive self-intersection
\cite[p.~1054, Lemma~4.2]{Bou02}. Let us now recall a few basic facts about the behavior of these 
notions under bimeromorphisms.

\begin{lemm}
\label{big}
Let $f\,:\,X\,\longrightarrow\, Y$ be a surjective, bimeromorphic morphism between two compact
complex manifolds. Let $[\alpha]\,\in\, H^{1,1}_{\partial \overline\partial}(Y,\,\R)$ and
$[\beta]\,\in\, H^{1,1}_{\partial \overline\partial}(X,\,\R)$ be cohomology classes.
\begin{enumerate}
\item If $[\beta]$ is big, then so is $f_*[\beta]$. 

\item If $[\alpha]$ is nef and has positive self-intersection, then $f^*[\alpha]$ is
also nef and has positive self-intersection.

\item If $[\alpha]$ is big and $X$ is K\"ahler, then $f^*[\alpha]$ is big.
\end{enumerate}
\end{lemm}

\begin{proof}
Let $\om_Y$ be a hermitian metric on $Y$, and let $\om_X$ be a hermitian metric on $X$ 
such that $\om_X\,\ge\, f^*\om_Y$.

Proof of 1: If $[\beta]$ is big, then there exists a closed positive current $T\,\in\, 
[\beta]$ and also a real number $\ep_0\,>\, 0$, such that $T\,\ge\, \ep_0\, \om_X \,\ge\, 
\ep_0 f^*\om_Y$. Then, $f_*T\in f_*[\beta]$ satisfies the condition $f_*T \,\ge\,\ep_0\, 
\om_Y$, and consequently $f_*[\beta]$ is big.

Proof of 2: If $[\alpha]$ is nef, then for every $\ep\,>\,0$, there exists a smooth form $\ome 
\,\in\, [\alpha]$ such that $\ome \,\ge\, -\ep \, \om_Y$. Therefore,
the pullback $f^*\ome \,\in\, f^*[\alpha]$ satisfies the condition
$$f^*\ome \,\ge\, -\ep \, f^*\om_Y \,\ge\, -\ep \, \om_X\, .$$ From this it follows that
$f^*[\alpha]$ is nef. Moreover, we have $$\int_X (f^*\alpha)^n \,=\, \int_Y\alpha^n\,>\,0\, .$$

Proof of 3: Take $[\alpha]$ to be big. Let $T\,\in\, [\alpha]$ be a K\"ahler current. In view
of Demailly's regularization theorem, we may assume that $T$ has analytic singularities (see
Section~\ref{nk}). Since $T$ can locally be written as $\alpha+dd^c \vp$ with
$\vp$ quasi-plurisubharmonic, we may set the pullback of $T$ to be $f^*T\,=\,
f^*\alpha+dd^c (\vp \circ f)$. Then, $f^*T$ is a positive current lying in the class
$f^*[\alpha]$, and $f^*T$ is a K\"ahler metric on a Zariski open set of
$X$. Now from \cite[p.~1057, Theorem~4.7]{Bou02} it follows that $f^*[\alpha]$ is big.
\end{proof}

\begin{defi}
\label{fuj}
A compact complex manifold $Y$ of dimension $n$ is said to be in Fujiki class $\mathcal 
C$ if there exist a compact K\"ahler manifold $X$ and a surjective meromorphic map
$$
f\,:\, X\,\longrightarrow \, Y
$$
(\cite{Fu, Fu2}, \cite[p.~50, Definition 3.1]{Va}). By \cite[p.~51, Theorem~5]{Va},
for any
$Y$ in Fujiki class $\mathcal C$, the above pair $(X,\, f)$ can be so chosen that the map
$f$ from the compact K\"ahler manifold $X$ is a bimeromorphism. In other words, a
Fujiki class $\mathcal C$ manifold admits a K\"ahler modification.
\end{defi}

\begin{rema}
\label{fujiki}
A manifold $Y$ in Fujiki class $\mathcal C$ admits a K\"ahler current (big class). Indeed, if $f\,:\,X\,
\longrightarrow \, Y$ is a K\"ahler modification, and $\om$ is a K\"ahler form on $X$, then $f_*\om$ is
a K\"ahler current on $Y$ (see proof of Lemma \ref{big} (1)).
In the opposite direction, it was proved in \cite[p.~1250, Theorem 0.7]{DP} that any
complex compact manifold $Y$ admitting a K\"ahler current is necessarily in Fujiki class $\mathcal C$. 
\end{rema}

\begin{coro}\label{cor1}
Let $Y$ be a compact complex manifold in Fujiki class $\mathcal C$. Let $\lbrack \alpha\rbrack
\,\in\, H^{1,1}_{\partial \overline\partial}(Y,\,\R)$ be a nef class. 
Then, $\lbrack \alpha\rbrack$ is big if and only if $[\alpha]$ has positive self-intersection.
\end{coro}

\begin{proof}
Let $f\,:\, X\,\longrightarrow \, Y$ be a K\"ahler modification. 

Assume that $[\alpha]$ has positive self-intersection. Then the pullback
$f^*\lbrack \alpha\rbrack$ is nef and has
positive self-intersection by Lemma~\ref{big} (2). Now Theorem \ref{DP} says that
$f^*\lbrack \alpha\rbrack$ is big. Therefore, Lemma~\ref{big} (1) implies that $[\alpha]$ is big. 

Conversely, assume that $[\alpha]$ is big. By Lemma~\ref{big} (3), the pullback $f^*[\alpha]$ is big as
well and if follows from \cite[p.~1054, Lemma~4.2]{Bou02} that $f^*[\alpha]$ has positive self-intersection. As
$\int_Y \alpha^n\,=\,\int_X (f^*\alpha)^n\,>\,0$, the result is proved.
\end{proof}

\begin{rema}\label{rem1}
Demailly and P\u{a}un conjectured the following (\cite[p.~1250, Conjecture~0.8]{DP}): If a complex compact 
manifold $Z$ possesses a nef cohomology class $\lbrack \alpha\rbrack$ which has positive
self-intersection, then $Z$ lies in the Fujiki class $\mathcal C$. This conjecture
would imply that a nef class $[\alpha]\,\in\, H^{1,1}_{\ddb}(Z,\,\R)$ on a compact complex manifold $Z$ is big
if and only if $[\alpha]$ has positive self-intersection (see Corollary \ref{cor1}).
\end{rema}

\subsection{Non-K\"ahler locus of a cohomology class}\label{nk}

\begin{defi}
Let $X$ be a complex manifold and $U\,\subset\, X$ an open subset.
\begin{enumerate}
\item A plurisubharmonic function (psh for short) $\vp$ on $U$ is said to have analytic singularities 
if there exist holomorphic functions $f_1,\, \ldots, \,f_r \,\in \, \mathcal O_X(U)$, a smooth function 
$\psi \,\in\, \mathcal C^{\infty}(U)$, and $a\,\in\, \R_+$, such that $$\vp\,=\, a 
\log(|f_1|^2+\ldots+|f_r|^2)+\psi\, .$$

\item Let $T$ be a closed, positive $(1,1)$-current on $X$. Then $T$ is said to have analytic 
singularities if it can be expressed, locally, as $T\,=\, dd^c\vp$, where
$\vp$ is a psh function with analytic singularities. The singular set for $T$ is denoted by 
$E_+(T)$; it is a proper Zariski closed subset of $X$.
\end{enumerate}
\end{defi}

It follows from the fundamental regularization theorems of Demailly, \cite{De}, that any big class 
$[\alpha]$ in a compact complex manifold $X$ contains a K\"ahler current $T$ with analytic 
singularities. Note that such a current $T$ is smooth on a non-empty Zariski open subset of $X$, and
$T$ induce a K\"ahler metric on this Zariski open subset. Following \cite{Bou04}, we define:

\begin{defi}\label{def1}
Let $X$ be a compact complex manifold, and let $[\alpha]\,\in\, H^{1,1}_{\partial \overline\partial}
(X,\, {\mathbb R})$ be a big cohomology class. The \textit{non-K\"ahler locus} of $[\alpha]$ is
$$E_{ nK}([\alpha])\,:=\, \bigcap_{T\in [\alpha]} E_{+}(T)\, ,$$
where the intersection is taken over all positive currents with analytic singularities. The
\textit{ample locus} $\mathrm{Amp}([\alpha])$ of $[\alpha]$ is the complement of the non-K\"ahler locus,
meaning
$$\mathrm{Amp}([\alpha])\, :=\, X\smallsetminus E_{nK}([\alpha])\, .$$
\end{defi}

Boucksom proved the following: Given any big class $[\alpha]\,\in\, H^{1,1}_{\partial 
\overline\partial}(X,\, {\mathbb R})$, there exists a positive current $T\,\in\, [\alpha]$ with 
analytic singularities such that $E_+(T)\,=\,E_{nK}([\alpha])$; in particular, $E_{nK}(\alpha)$ is a 
proper Zariski closed subset of $X$ (see \cite[p.~59, Theorem 3.17]{Bou04}).

The following Proposition builds upon a result of Collins and Tosatti \cite[p.~1168, Theorem~1.1]{CT} 
showing that the non-K\"ahler locus of a nef and big class $[\alpha]$ on a Fujiki manifold 
$X$ coincides with its null locus $\mathrm{Null}([\alpha])$ defined as the reunion of all 
irreducible subvarieties $V\,\subseteq\, X$ such that $\int_{V}\alpha^{\dim V}\,=\,0$.

\begin{prop}\label{null}
Let $f\,:\,X\,\longrightarrow\, Y$ be a bimeromorphic morphism between two compact complex manifolds 
belonging to the Fujiki class $\mathcal C$. Let $[\alpha]\,\in\, H^{1,1}_{\ddb}(X,\,\R)$ be a nef 
and big cohomology class. Then $$E_{nK}(f^*[\alpha])\,= \,f^{-1}(E_{nK}([\alpha])) \cup 
\mathrm{Exc}(f)\, ,$$ where $\mathrm{Exc}(f)$ is the exceptional locus of $f$, i.e., the singular
locus of the Jacobian of $f$.
\end{prop}

\begin{proof}
By \cite[p.~1168, Theorem~1.1]{CT}, it is enough to prove the analogous result for null loci. Let 
$n\,:=\,\dim_{\C} X$.

Let $E\,\subset\, \mathrm{Exc}(f)$ be an irreducible component; it has dimension $n-1$, while 
$\dim f(E)\,\le\, n-2$. Therefore, $\int_E (f^*\alpha)^{n-1}\,=\,\int_{f(E)} \alpha^{n-1}\,=\,0$ and $E 
\,\subset\, E_{nK}(f^*[\alpha])$. Next, if $V\,\subset\, E_{nK}([\alpha])$ is a $k$-dimensional subvariety not 
included in $f(\mathrm{Exc}(f))$, let $\widetilde V$ be its strict transform. We have 
$f^{-1}(V) \,=\,\widetilde{V} \cup F$ with $F\,\subset\, \mathrm{Exc}(f)$ and $f$ inducing a 
bimeromorphic morphism $f|_{\widetilde V}\,:\,\widetilde{V}\,\longrightarrow \, V$. From the 
identity $\int_{\widetilde V} (f^*\alpha)^{k} \,=\,\int_{V} \alpha^{k}$ it follows that
$$\widetilde{V}\,\subset \,E_{nK}(f^*[\alpha])\, .$$ 
In summary, $ f^{-1}(E_{nK}([\alpha])) \cup \mathrm{Exc}(f) \subset E_{nK}(f^*[\alpha])$.

Now, let $W\subset E_{nK}(f^*[\alpha])$ be an irreducible $k$-dimensional subvariety not 
included in $\mathrm{Exc}(f)$. The morphism $f$ induces a bimeromorphic morphism $$f|_W\,:\,W\,
\longrightarrow\, f(W)\, ,$$ and $\int_{f(W)} \alpha^k\,=\, \int_W (f^*\alpha)^k\,=\,0$. Therefore,
$f(W)\,\subset\, E_{nK}([\alpha)])$, which completes the proof of the Proposition.
\end{proof}

\subsection{Singular Ricci-flat metrics}
\label{RFK}
Let $Y$ be a compact complex manifold of dimension $n$ such that the first Chern class $c_1(T_Y)$ 
vanishes in $H^2(Y,\, \R)$. We assume that $Y$ admits a big cohomology class $[\alpha]\,\in\, 
H^{1,1}_{\ddb}(Y,\,\R)$. In particular, $Y$ lies in Fujiki class $\mathcal C$ (see 
Remark~\ref{fujiki}). We fix once and for all a K\"ahler modification
$$f\,:\,X\,\longrightarrow\, Y\, .$$

\noindent
It follows from Lemma~\ref{big} (3) that the class $f^*[\alpha]$ is big. The Jacobian of $f$ induces the following identity
\begin{equation}\label{e2}
K_X\,=\,f^*K_Y+E\, ,
\end{equation}
where
\begin{equation}\label{e3}
E\,=\, \sum_{i=1}^r a_i E_i
\end{equation}
is an effective $\Z$-divisor on $X$ contracted by $f$, meaning $\mathrm{codim}_Y f(E) \,\ge\, 2$, and
each $E_i$ is irreducible. More precisely, $\mathrm{Supp}(E)$ coincides with the exceptional locus of
$f$, that is the complement of the locus of points on $X$ in a neighborhood of which $f$ induces a
local biholomorphism.

Since $c_1(K_Y)\,=\,0 \,\in\, H^2(Y,\,\R)$, it follows from \eqref{e2} that $K_X$ is numerically
equivalent to the effective divisor $E$. Moreover, the class $f^*[\alpha] \,\in\, H^{1,1}(X,\, \R)$ is big
by Lemma~\ref{big}\, (3). From \cite[p.~200, Theorem~A]{BEGZ} we know that there exists a unique
closed positive current
\begin{equation}\label{dt}
T\,\in\, f^*{ \lbrack \alpha \rbrack }
\end{equation}
with finite energy such that $$-\mathrm{Ric}(T)\,=\,[E]\, .$$ 

In terms of Monge--Amp\`ere equations, this means that the non-pluripolar Monge--Amp\`ere measure of
$T$, denoted by $\la T^n\ra$, can be expressed as $$\la T^n \ra \,=\, |s_E|_h^2dV\, ,$$
where
\begin{itemize}
\item $h$ is a smooth hermitian metric on $\mathcal O_X(E)$ with Chern curvature tensor denoted by 
$\Theta_h(E)$,

\item $s_E\,\in\, H^0(X,\mathcal O_X(E))$ satisfies the condition $\mathrm{div}(s_E)\,=\,E$, and

\item $dV$ is the smooth volume form on $X$ satisfying $\Ric(dV)\,=\,-\Theta_h(E)$ and normalized such that
$\int_X |s_E|_h^2dV\,=\,\mathrm{vol}(f^*\alpha)\,=\,\mathrm{vol}(\alpha)$.
\end{itemize}

\begin{propdef}
\label{metric}
Let $Y$ be a compact complex manifold with trivial first Chern class endowed with a big cohomology class $[\alpha] \in H^{1,1}_{\d\db}(Y,\,\R)$. Let $f\,:\,X\,\longrightarrow\, Y$ be a K\"ahler modification, and let $T\in f^*[\alpha]$ be the associated singular Ricci-flat metric. Then 
\begin{enumerate}
 \item There exists a closed, positive $(1,1)$-current $\omega\in [\alpha]$ on $Y$ such that $T=f^*\om$. 
 \item The current $\omega$ is independent of the choice of the K\"ahler modification of $Y$.
 \end{enumerate}
 We call $\omega$ the singular Ricci-flat metric in $[\alpha]$.
\end{propdef}

\begin{proof}
For the first item, let us write $T=f^*\alpha+dd^c u$ for some $(f^*\alpha)$-psh function $u$ on $X$. The 
restriction of $u$ to any fiber of $f$ is psh, but the fibers of $f$ are connected. By the maximum principle, 
$u$ is constant on the fibers of $f$, hence can be written as $u\,=\,\pi^*v$ for some $\alpha$-psh function $v$ 
on $Y$. Then, $T\,=\,f^*\omega$ with $\omega\,=\,\alpha+dd^c v$.

Now, let $f'\,:\,X'\,\longrightarrow\, Y$ be another K\"ahler modification, and let $T'\,:=\,(f')^*\om'$ be the
associated singular Ricci-flat current. Let $Z$ be a desingularization of $X'\times_Y X$, so that we have the
following Cartesian square where all maps are bimeromorphic
 $$
\xymatrix{
     Z \ar[rr]^{g'}  \ar[d]_{g} && X  \ar[d]^{f} \\
      X' \ar[rr]^{f'}       && Y,
  }
  $$
Let $h\,:=\,f\circ g'\,=\,f'\circ g$. By \cite[p.~201, Theorem B]{BEGZ} and \cite[Proposition~1.12]{BEGZ}, both
currents $g^*T'$ and $(g')^*T$ have minimal singularities in $h^*[\alpha]$ and are solutions of the non-pluripolar
Monge-Amp\`ere equation $-\Ric(\cdotp)\,=\,[K_{Z/Y}]$. By uniqueness of the solution of that
equation \cite[Theorem~A]{BEGZ}, if follows that $h^*\om\,=\,h^*\om'$. Taking the direct image of the
previous equality, it is deduced that $\om\,=\,\om'$.
\end{proof}

The regularity properties of the current $T$ (or $\omega$) are quite mysterious in general. For instance, it is 
not known whether $T$ is smooth on a Zariski open set. However, things become simpler once we assume 
additionally that the cohomology class $[\alpha]$ is nef.

\begin{propdef}
\label{om}
In the set-up of Proposition-Definition~\ref{metric}, assume additionally that the cohomology class $[\alpha] $ is nef. Then, the singular Ricci-flat current $\omega\in [\alpha]$ is smooth on a non-empty Zariski open subset of $Y$. 

We define $\Omega$ to be the largest Zariski open subset in restriction of which $\om$ is a genuine K\"ahler 
form. Then
$$\emptyset \neq f(\mathrm{Amp}(f^*[\alpha]))\, \subseteq \,\Omega\, \subseteq\, \mathrm{Amp}([\alpha])$$
for any K\"ahler modification $f\,:\,X\,\longrightarrow\, Y$ (see Definition~\ref{def1}). 
\end{propdef}

\begin{proof}
Let $f\,:\,X\,\longrightarrow\,Y$ be a K\"ahler modification. As $[\alpha]$ is nef, $f^*[\alpha]$ is nef as well (see 
Lemma~\ref{big} (2)) and the proof of \cite[p.~201, Theorem C]{BEGZ} applies verbatim to give that 
$T$ is smooth outside $E_{nK}(f^*[\alpha]) \smallsetminus \mathrm{Supp}(E)$. Note that 
Proposition~\ref{null} shows that this locus is just $E_{nK}(f^*[\alpha])$. In particular, this shows that the current $\omega \in [\alpha]$ on $Y$ is a genuine Ricci-flat K\"ahler metric on the non-empty Zariski open set
$$f(\mathrm{Amp}(f^*[\alpha]))\, \subseteq \, \mathrm{Amp}([\alpha]),$$
which concludes the proof. 
\end{proof}

We will need a refinement of the result above which will be explained next.

Let $t,\,\ep\,>\, 0$; by Yau's theorem (Calabi's conjecture) \cite{Ya}, there exists a unique 
K\"ahler metric $\omte \,\in\, f^*{\lbrack \alpha \rbrack } +t[\om_X]$ solving the equation 
\begin{equation}
\label{regu}
\mathrm{Ric}(\omte)\,=\, t\om_X-\theta_{\ep}\, ,
\end{equation} where $\theta_{\ep} \,\in\, c_1(E)$ is a 
regularization of the current defined by integration along $E$. For instance, such a smooth closed 
$(1,1)$--form $\theta_{\ep}$ can be constructed as follows: pick smooth hermitian metrics $h_i$ on 
$\mathcal O_X(E_i)$ (see \eqref{e3}), and choose a holomorphic section $s_i\,\in\, H^0(X,\mathcal O_X(E_i))$ for each 
$i$ cutting out $E_i$ in the sense that $\div(s_i)\,=\,E_i$; then set $$\theta_{\ep}\,=\, 
\sum_{i=1}^r a_i\Big(\Theta_{h_i}(E_i)+dd^c \log (|s_i|^2_{h_i}+\ep^2)\Big)\, ,$$ where 
$\Theta_{h_i}(E_i)$ is the Chern curvature of the hermitian line bundle $(\mathcal O_X(E_i), h_i)$. 

Now, the proof of \cite[p.~201, Theorem C]{BEGZ} can be adapted without any significant changes to 
obtain the following result.

\begin{thm}[{\cite[Theorem C]{BEGZ}}] 
\label{convergence}
In the set-up of Proposition-Definition~\ref{metric}, assume furthermore that 
$[\alpha]$ is nef, and let $\omte\,=\, f^*{\lbrack \alpha \rbrack } +t[\om_X]$ be the K\"ahler 
form solving Eq.~\eqref{regu}.
When $t\,\to\, 0$ and $\ep\,\to\, 0$, the form $\omte$ converges to the current $T$ in \eqref{dt} in the
weak topology of currents, and also in the $\mathcal C^{\infty}_{\rm loc}(\mathrm{Amp}(f^*[\alpha]))$ topology.
\end{thm}

\subsection{Flatness of tensors}
It will be convenient in this section to refer to the following set-up. 
\begin{set} 
\label{set}
Let $Y$ be a compact complex manifold with trivial first Chern class endowed with a nef and big cohomology class $[\alpha] \,\in\, H^{1,1}_{\d\db}(Y,\, \R)$. We denote by $\om$ the singular Ricci-flat metric from Proposition-Definition~\ref{metric}, and we let $\Om$ be its smooth locus; see Proposition-Definition~\ref{om}. 
\end{set}

\begin{thm}\label{Bochner principle}
In Setting~\ref{set}, let $\tau \,\in\, H^0(Y,\,{T_Y^*}^{\otimes p})$ with $p\, \geq \, 0$.

\noindent
Then, the restriction $\tau|_{\Om}$ to $\Om$ in \eqref{om} is parallel
with respect to the Ricci-flat K\"ahler metric $\om|_{\Om}$.
\end{thm}

\begin{rema}\label{chern}
Since $c_1(T_Y)\,=\,0$ in $H^2(Y,\, \R)$, and $Y$ is in Fujiki class $\mathcal C$, 
from \cite[Theorem 1.5]{To} we know that the holomorphic line bundle $K_Y$ is of finite order. 
Therefore, there exists a finite unramified cover $\pi\,:\,\widehat{Y}\,\longrightarrow\, Y$ such that $K_{\widehat Y}$ is 
holomorphically trivial. 
\end{rema}

The previous remark leads to the following application of Theorem~\ref{Bochner principle}, valid for \textit{any} holomorphic tensors. 

\begin{coro}\label{determination}
In Setting~\ref{set}, let $\tau \,\in\, H^0(Y,\,{T_Y}^{\otimes p}\otimes {T_Y^*}^{\otimes q})$ with $p,q\, \geq \, 0$. Then, the restriction $\tau|_{\Om}$ to $\Om$ in \eqref{om} is parallel
with respect to the Ricci-flat K\"ahler metric $\om|_{\Om}$.

\noindent
In particular, the evaluation map 
$$\begin{matrix}
\mathrm{ev}_y: &H^0(Y, T_Y^{\otimes p} \otimes {T_Y^*}^{\otimes q})& \longrightarrow &
(T_Y^{\otimes p} \otimes {T_Y^*}^{\otimes q})_y \\
& \tau &\longmapsto & \tau(y)
\end{matrix}$$
is injective for all $y\,\in\, \Omega$ and all integers $p,\, q\, \geq \, 0$.
\end{coro}

\begin{proof}[Proof of Corollary~\ref{determination}]
Let $\tau$ be any holomorphic tensor on $Y$, not necessarily contravariant. Let $\pi\,:\,\widehat{Y} 
\,\longrightarrow\, Y$ be the finite unramified cover from Remark~\ref{chern}. As $\pi$ is a local 
biholomorphism, the pull-back $\pi^*\tau$ on $\widehat Y$ is well-defined. We can interpret $\pi^*\tau$ as a 
holomorphic contravariant tensor on $\widehat Y$. Indeed, a holomorphic trivialization of $K_{\widehat Y}$ 
produces a holomorphic isomorphism $T_{\widehat Y} \,\simeq\,\bigwedge^{n-1}T^{*}_{\widehat Y}$, where 
$n\,=\,\dim_{\mathbb C}Y$.

Now, let $f\,:\,X\,\longrightarrow\, Y$ be a K\"ahler modification, and let us set $T\,:=\,
f^*\om$ and $\Omega'\,:=\,f(\mathrm{Amp}(f^*[\alpha]))\,\subset\, \Omega$. As $\Omega'$ is dense in $\Omega$ for
the usual topology, and both $\tau$ and $\omega$ are smooth on $\Omega$, it suffices to prove that
$\tau|_{\Omega'}$ is parallel with respect to $\om|_{\Omega'}$. Since $f$ is an isomorphism over
$\Omega'$, we may pull back $\tau|_{\Omega'}$ by $f$ over this locus. This way we reduce the question to
proving that $f^*\tau|_{\Omega'}$ is parallel with respect to $T|_{f^{-1}(\Omega')}$.

Let $\widehat{X}\,:=\,X\times_{Y} \widehat{Y}$, so that we have a Cartesian square
 $$
  \xymatrix{
    \widehat{X} \ar[rr]^{\widehat f}  \ar[d]_{\widehat \pi} && \widehat{Y}  \ar[d]^{\pi} \\
      X \ar[rr]^{f}       && Y,
  }
  $$
The morphism $\widehat\pi$ is finite unramified (in particular, $\widehat X$ is compact K\"ahler) and $\widehat 
f$ is birational. By the observation at the beginning of the proof, we may apply Theorem~\ref{Bochner 
principle} to show that $\widehat{f}^*\pi^*\tau$ is parallel with respect to the singular Ricci-flat K\"ahler 
metric $T_{\widehat X}\,\in\, \widehat{f}^*\pi^*[\alpha]$ on the locus $\mathrm{Amp}(\widehat{f}^*\pi^*[\alpha])$. 
As $T_{\widehat X} \,\in\, \widehat{\pi} ^*f^*[\alpha]$, the functoriality
property of K\"ahler-Einstein metrics with 
respect to finite morphisms (see e.g. \cite[Proposition~3.5]{GGK}) shows that $T_{\widehat 
X}\,=\,\widehat{\pi}^* T$. Over $\Omega'$, the following identity holds $$\widehat{f}^*\pi^*\tau|_{\Omega'}\,=\, 
\widehat{\pi} ^*f^*\tau|_{\Omega'}\, .$$ Therefore, $\widehat{\pi} ^*f^*\tau|_{\Omega'}$ is parallel with respect to 
$\widehat{\pi}^* T|_{f^{-1}(\Omega')}$, hence $f^*\tau|_{\Omega'}$ is parallel with respect to 
$T|_{f^{-1}(\Omega')}$. The Corollary now follows easily.
\end{proof}

Let us now prove Theorem \ref{Bochner principle}. The arguments and computations in the proof are
extensively borrowed from \cite{Gue} (see also \cite{CP1}).

\begin{proof}[Proof of Theorem~\ref{Bochner principle}]
Let us fix a K\"ahler modification $f\,:\,X\,\longrightarrow\, Y$, and let us set $\Omega':=f(\mathrm{Amp}(f^*[\alpha])) \subset \Omega$. By the same arguments as in the proof of Corollary~\ref{determination}, it is sufficient to prove that $\tau|_{\Omega'}$ is parallel with respect to $\om|_{\Omega'}$ or, equivalently, that the restriction of $\sigma\,:=\,f^*\tau \,\in \, H^0(X, \,{T_X^*}^{\otimes p})$ to $f^{-1}(\Omega')$ is parallel with respect to $(f^*\om)|_{f^{-1}(\Omega')}$.

Let $\mathcal E\,:=\,{T_X^*}^{\otimes p}$, and let $h\,=\,|\cdotp \!|$ be the hermitian metric on
$\mathcal E$ induced by the K\"ahler metric $\omte$ introduced in Eq.~\eqref{regu}. Let $D\,=\,D'+\db$ be the
corresponding Chern connection for $(\mathcal E,\,h)$.The curvature of this Chern connection for $(\mathcal E,\,h)$
will be denoted by $\Theta_h( \mathcal E)$. The following holds:
\begin{equation}\label{ineq0} 
dd^c \log(\ud)\,=\,
\frac{1}{\ud}\left( |D'\sigma|^2-\frac{|\la D'\sigma,\sigma\ra |^2}{\ud}-\la
\Theta_h(\mathcal E)\sigma,\sigma\ra \right)\, .
\end{equation}
Wedging \eqref{ineq0} with $\omte^{n-1}$ and then integrating it on $X$ yields:
$$\int_X \frac{\la\Theta_h(\mathcal E)\sigma,\sigma\ra}{\ud} \wedge \omte^{n-1}
\,=\, \int_X \frac{1}{\ud}\left( |D'\sigma|^2-\frac{|\la D'\sigma,\sigma\ra |^2}{\ud} \right)
\wedge \omte^{n-1}\, .$$
Since $|\la D'\sigma,\sigma\ra | \,\le\, |D' \sigma| \cdotp |\sigma|$, we obtain the inequality
\begin{equation}
\label{ineq00}
\int_X \frac{\la\Theta_h(\mathcal E)\sigma,\sigma\ra}{\ud} \wedge \omte^{n-1}
\,\ge\,\int_X \frac{ |D'\sigma|^2}{(\ud)^2} \wedge \omte^{n-1}\, .
\end{equation}

\vspace{3mm}

First let us introduce a notation: if $V$ is a complex vector space of dimension $n$,
$1\le p\le n$ is an integer, and $t \,\in \,\mathrm{End}(V)$, then we denote by $t\rr$ the
endomorphism of $V^{\otimes p}$ defined by
\[t\rr(v_1 \otimes \cdots \otimes v_p) \,:=\,
\sum_{i=1}^p v_1 \otimes \cdots\otimes v_{i-1} \otimes t(v_i) \otimes v_{i+1} \otimes
\cdots \otimes v_p\, .\]
it may be noted that if $V$ has an hermitian structure, and if $t$ is hermitian semipositive, then
$t \rr$ is also hermitian semipositive for the hermitian structure on $V^{\otimes p}$ induced
by the hermitian structure on $V$; also, the inequality $\tr (t\rr) \,\le\, n^p \, \tr( t)$ holds. 

Now we can easily check the following identity:
\[n \Theta_h(\mathcal E ) \wedge \omte^{n-1} \,=\, -(\sharp \Ric \om)\rr \,\om^ n\, , \]
where $\sharp \Ric \om$ is the endomorphism of $T_X^*$ induced by $\Ric \omte$ via $\omte$. As 
$\Ric \omte\, =\,-\theta_{\ep}$, we deduce that 
$$\int_X \frac{\la\Theta_h(\mathcal E)\sigma,\sigma\ra}{\ud} \wedge \omte^{n-1}
\,=\,-\int_X \frac{\la(\sharp \theta_{\ep})\rr \sigma, \sigma \ra}{\ud} \om^n$$
(the operator $\sharp$ is defined above).

\vspace{3mm}
We can write $\theta\,= \,\sum a_i \theta_{i,\ep}$, where $\theta_{i,\ep}\,:=\,
a_i\left(\frac{\ep^2 |D's_i|^2}{(|s_i|^2+\ep^2)^2}+
\frac{\ep^2 \Theta_{h_i}(E_i)}{|s_i|^2+\ep^2}\right) $. In order to simplify the notation, we
drop the index $i$ and set 
$$\beta \,=\, \frac{\ep^2 |D's|^2}{(|s|^2+\ep^2)^2}\ \ \text{ and }\ \ \gamma \,=\,
\frac{\ep^2\Theta_{h_i}(E_i)}{|s|^2+\ep^2}\, ,$$ so that $\theta_{i,\ep}\,=\,a(\beta+\gamma)$; remember
that these forms are smooth as long as $\ep\,>\,0$.

Let us start with $\gamma$: there exists a
constant $C\,>\,0$ such that $$\pm \gamma \,\le\, C\ep^ 2/(|s|^2+\ep^ 2) \, \om_X\, .$$ As both
of the two operations $\sharp$ and $p$-th tensor power preserve positivity, we get that
$$\pm (\sharp \gamma)\rr \omte^n \,\le\, C\ep^ 2/(|s|^2+\ep^ 2) \, (\sharp \om_X)\rr \omte^n\, .$$
But $\sharp \om_X$ is a positive endomorphism whose trace is $\tr_{\omte}\om_X$, and therefore
we have $(\sharp \om_X)\rr \,\le\, n^p \tr_{\omte}(\om_X)\cdot \mathrm{Id}$. Consequently, 
\begin{eqnarray*}
\frac{\pm \la(\sharp \gamma)\rr \sigma,\sigma\ra}{\ud} \, \omte^n &\le& \frac{C\ep^2}{|s|^2+\ep^2} \cdot \frac{|\sigma|^2}{\ud}\, \om_X \wedge \omte^{n-1}\\
&\le & \frac{C\ep^2}{|s|^2+\ep^2} \, \om_X \wedge \omte^{n-1}
\end{eqnarray*}
for some $C\,>\,0$ which is independent of
$t$ and $\ep$. From Lemma \ref{lem2} below and using the dominated convergence theorem,
we deduce that the integral 
\[ \int_X \frac{\la(\sharp \gamma)\rr \sigma,\sigma\ra \, }{\ud}\omte^ n\]
converges to $0$ when $\ep$ goes to zero. \\

We now have to estimate the term involving $\beta$. We know that $\beta$ is non-negative, so 
$(\sharp \beta)\rr \omte^n \,\le\, n^{p+1}\, \beta \wedge \omte^ {n-1} \, \mathrm{Id}$, and hence 
\begin{eqnarray*}
0 \le \int_X \frac{\la(\sharp \beta)\rr \sigma,\sigma \ra }{\ud} \, \omte^n &\le &
C \int_X \frac{|\sigma|^2}{\ud} \cdot \beta \wedge \omte^{n-1} \\
& \le & C \int_{X} \beta \wedge \omte^{n-1} \\
&=& C\left(\int_{X} (\beta+\gamma) \wedge \omte^{n-1} -\int_{X} \gamma \wedge \omte^{n-1} \right)\\
&=& C \left( \{\theta\}\cdotp \{\om\}^{n-1}-\int_{X} \gamma \wedge \omte^{n-1} \right)\, .
\end{eqnarray*}
We already observed that the second integral converges to $0$ when $\ep \,\to\, 0$. As for the first 
term, it is cohomological (independent of $\ep$), equal to $t^{n-1}(E_i \cdotp [\om_X]^{n-1})$ 
since $f^*[\alpha]$ is orthogonal to $E_i$, and thus it converges to $0$ as $t$ goes to $0$. \\

Now recall \eqref{ineq00}:
$$\int_X \frac{\la\Theta_h(\mathcal E)\sigma,\sigma\ra}{\ud} \wedge \omte^{n-1} \,\ge\,
\int_X \frac{|D'\sigma|^2}{(\ud)^2} \wedge \omte^{n-1} \,\ge\, 0\, . $$
When $\ep$ and $t$ both go to $0$, $\om_{t,\ep}$ converges weakly to $f^*\om$. Moreover, $f^*\om$ is
smooth on $f^{-1}(\Omega')$ (defined at the beginning of the proof), and the convergence is smooth on the compact subsets of $\Om'$ by Theorem~\ref{convergence}.
Therefore, by Fatou lemma, we deduce that $D'\sigma\,=\,0$ on this locus, which was required
to be shown (here $D'$ denotes the Chern connection on $\mathcal E$
associated to $f^*\om$ on this open subset $\Omega'$).
\end{proof}

The proof of Theorem \ref{Bochner principle} involved the following result, which is proved
in much greater generality in \cite[Lemma~3.7]{Gue}. We provide a simpler proof more suited to the present
set-up. 

\begin{lemm}
\label{lem2}
For every fixed $t\,>\, 0$, and any section $s\, \in\, {\mathcal O}_X(E_i)$ for
some component $E_i$ of $E$, the integral
\[\int_{X} \frac{\ep^ 2}{|s|^2+\ep^ 2} \, \om_X \wedge \omte^{n-1}\]
converges to $0$ when $\ep$ goes to $0$. 
\end{lemm}

\begin{proof}
By \cite[p.~360, Proposition~2.1]{Ya}, there is a constant $C_t>0$ independent of 
$\ep>0$ such that $\omte \,\le\, C_t \, \om_X$. The lemma thus follows from Lebesgue's dominated 
convergence theorem.
\end{proof}

\section{Geometric structures on manifolds in class $\mathcal C$}\label{geom struct}

In this section we give two applications of Theorem \ref{Bochner principle} for manifolds in Fujiki 
class $\mathcal C$ bearing a holomorphic geometric structure.

\subsection{Holomorphic geometric structures}

Let us first recall the definition of a holomorphic (rigid) geometric structure as given in 
\cite{DG, Gr}.

Let $Y$ be a complex manifold of complex dimension $n$ and $k \,\geq\, 1$ an integer. Denote by 
$R^k(Y)\, \longrightarrow \,Y$ the holomorphic principal bundle of $k$-frames of $Y$: it is the bundle of 
$k$-jets of local holomorphic coordinates on $Y$. Recall that the structure group of $R^k(Y)$ 
is the group $D^k$ of $k$-jets of local biholomorphisms of $\mathbb{C}^n$ fixing the origin. 
This $D^k$ is a complex algebraic group.

\begin{defi}\label{hgs}
A {\it holomorphic geometric structure} $\phi$ (of order $k$) on $Y$ is a holomorphic 
$D^k$--equivariant map from $R^k(Y)$ to a complex algebraic manifold $Z$ endowed with 
an algebraic action of the algebraic group $D^k$.
\end{defi}

A holomorphic geometric structure $\phi$ as in Definition \ref{hgs} is said to be of {\it affine
type} if $Z$ in Definition \ref{hgs} is a complex affine manifold.

Notice that holomorphic tensors are holomorphic geometric structures of affine type of order 
one. Holomorphic affine connections are holomorphic geometric structures of affine type of order 
two \cite{DG}. In contrast, while holomorphic foliations and holomorphic projective connections 
are holomorphic geometric structure in the sense of Definition \ref{hgs}, they are not of 
affine type.

A holomorphic tensor which is the complex analog of a Riemannian metric is defined in the 
following way:

\begin{defi}
A {\it holomorphic Riemannian metric} on a complex manifold $Y$ of complex
dimension $n$ is a holomorphic section
$$
g\, \in\, H^0(Y,\, \text{S}^2((T_Y)^*))\, ,
$$
where $\text{S}^i$ stands for the $i$-th symmetric product,
such that for every point $y\, \in\, Y$ the complex quadratic form $g(y)$ on $T_yY$
is of (maximal) rank $n$. 
\end{defi}

Take a holomorphic Riemannian manifold $(Y,\, g)$ as above.
The real part of $g$ is a pseudo-Riemannian metric $h$ of signature $(n,n)$ on the real manifold
of dimension $2n$ underlying the complex manifold $Y$.

As in the set-up of (pseudo-)Riemannian manifolds, there exists a unique torsion-free holomorphic connection 
$\nabla$ on the holomorphic tangent $T_Y$ such that $g$ is parallel with respect to $\nabla$. It 
is called the Levi-Civita connection for $g$.

Given $(Y,\, g)$ as above, consider the curvature of the holomorphic Levi-Civita connection 
$\nabla$ for $g$. This 
curvature tensor vanishes identically if and only if $g$ is locally isomorphic to the standard 
flat (complex Euclidean) model $(\mathbb C^n,\, dz_1^2 + \ldots + dz_n^2)$. In this flat case the real 
part $h$ of $g$ is also flat and it is locally isomorphic to $(\mathbb R^{2n},\, dx_1^2+ \ldots 
+dx_n^2-dy_1^2- \ldots- dy_n^2)$. For more details about the geometry of holomorphic Riemannian 
metrics the reader is referred to \cite{Du2, Gh2}.

A natural notion of (local) infinitesimal symmetry is the following (the
terminology comes from the standard Riemannian setting).

\begin{defi}
A (local) holomorphic vector field $\theta$ on $Y$ is a (local)
{\it Killing field} for a holomorphic geometric structure of order $k$
$$\phi \,:\, R^k(Y) \,\longrightarrow\, Z$$ if the flow for the canonical lift of
$\theta$ to $R^k(Y)$ preserves each of the fibers of the map $\phi$.
\end{defi}

Consequently, the (local) flow of a Killing vector field for $\phi$ preserves $\phi$. It is evident 
that the Killing vector fields for $\phi$ form a Lie algebra with respect to the operation of Lie 
bracket.

The holomorphic geometric structure $\phi$ is called {\it locally homogeneous} on an open subset 
$\Om$ of $Y$ if the holomorphic tangent bundle $T_Y$ is spanned by local Killing vector fields of 
$\phi$ in the neighborhood of every point in $\Om$. This implies that for any pair of points $o 
\, ,o' \,\in\, \Om$, there exists a (local) biholomorphism, from
a neighborhood of $o$ to a neighborhood of $o'$, that preserves $\phi$ and also sends $o$ to 
$o'$.

A holomorphic geometric structure $\phi$ is called {\it rigid} of order $l$ in the sense of 
Gromov, \cite{Gr}, if any local biholomorphism $f$ preserving $\phi$ is determined uniquely by
the $l$--jet of $f$ at any given point.

Holomorphic affine connections are rigid of order one in the sense of Gromov \cite{DG, Gr}. Their 
rigidity comes from the fact that local biholomorphisms fixing a point and preserving a 
connection actually linearize in exponential coordinates around the fixed point, so they are 
completely determined by their differential at the fixed point.

A holomorphic Riemannian metric $g$ is also a rigid holomorphic geometric structure, because 
local biholomorphisms preserving $g$ also preserve the associated Levi-Civita connection. In
contrast, holomorphic symplectic structures and holomorphic foliations are non-rigid geometric 
structures~\cite{DG, Gr}.

\subsection{A criterion for local homogeneity}

\begin{thm}\label{ouvert dense}
Let $Y$ be a compact complex manifold in Fujiki class 
$\mathcal C$ with trivial first Chern class ($c_1(T_Y)\,=\,0$ in $H^2(Y,\, \R)$). Assume that there 
exists a nef cohomology class $\lbrack \alpha \rbrack \,\in\, H^{1,1}(Y,\,\R)$ of positive 
self-intersection. Then the following two hold:
\begin{enumerate}
\item There exists a non-empty Zariski open subset $\Om\, \subset\, Y$ such that any holomorphic
geometric structure of affine type on $Y$ is locally homogeneous on $\Om$.

\item If $Y$ admits a rigid holomorphic geometric structure of affine type, then the 
fundamental group of $Y$ is infinite.
\end{enumerate} 
\end{thm}

\begin{proof} (1): In view of Corollary \ref{cor1}, the assumptions of Theorem~\ref{Bochner principle} are satisfied by $Y$.
Corollary~\ref{determination} implies that there exists a Zariski open 
subset $\Om\,\subset\, Y$ such that any holomorphic tensor on $Y$ is parallel with respect to 
some K\"ahler metric on $\Om$. In particular, any holomorphic tensor on $Y$ vanishing at some 
point of $\Om$ must be identically zero. Then Lemme 3.2 in 
\cite[p.~565]{Du} gives that any holomorphic geometric structure of affine type on $Y$ is locally
homogeneous on $\Om$.

(2): To prove by contradiction, assume that the fundamental group of $Y$ is finite. Substituting the 
universal cover of $Y$ in place of $Y$, and considering the pull-back, to the universal cover, of the 
geometric structure on $Y$, we may assume $Y$ in the theorem to be simply connected.

Now $K_Y$ is 
trivial because $Y$ is simply connected; see Remark \ref{chern}. A result, first proved by Nomizu in 
the Riemannian setting \cite{No}, and subsequently generalized by Amores \cite{Am} and Gromov 
\cite{Gr}, gives the following: the condition that $Y$ is simply connected implies that any local 
(holomorphic) Killing field of a rigid holomorphic geometric structure on $Y$ extends to a global 
(holomorphic) Killing field (see also a nice exposition of it in \cite{DG}).

In particular, using 
statement (1) in the theorem, we obtain that at any point of $z\, \in\, \Om$, the fiber $T_z \Om$ of 
the holomorphic tangent bundle $T_Y$ is spanned by globally defined holomorphic vector fields on $Y$.
It was noted in the proof of (1) that any holomorphic tensor on $Y$ that vanishes at some
point of $\Om$ must be identically zero. Combining these we conclude that there are
$n$ global holomorphic vector fields on $Y$, where $n\,=\,\dim_{\mathbb C}Y$, that span $T_\Om$.

Fix $n$ global holomorphic vector fields $X_1,\, \cdots,\, 
X_n$ on $Y$ that span $T_\Om$. Also, fix a nontrivial holomorphic section $vol$ of the trivial
canonical bundle $K_Y$. Then $vol(X_1, \cdots, X_n)$ is a global holomorphic function on $Y$. 
This function must be constant (by the maximum principle) and nonzero at points in $\Om$. Since
$vol(X_1, \cdots, X_n)$ is nowhere vanishing on $Y$, 
it follows that $X_1,\, \cdots, \, X_n$ span the holomorphic tangent bundle $T_Y$ at all points of 
$Y$.

In other words, $Y$ is a parallelizable manifold. Hence by a theorem of Wang, \cite[p.~774,
Theorem 1]{Wa}, 
the complex manifold $Y$ must be biholomorphic to a quotient of a connected complex Lie group by 
a co-compact lattice in it. In particular, $Y$ is not simply connected. This gives the contradiction
that we are seeking.
\end{proof}

\subsection{A non-affine type example}

It should be mentioned that statement (1) in Theorem \ref{ouvert dense} is not valid in general 
for holomorphic geometric structures of non-affine type. To see such an example, first recall 
that Ghys constructed in \cite{Gh} codimension one holomorphic foliations on complex tori which 
are not translation invariant. Such a foliation can be obtained in the following way. Consider a 
complex torus $T\,=\,\C^{n} / \Lambda$, with $\Lambda$ a lattice in $\C^{n}$ and assume that 
there exists a linear form $\widetilde{\pi }\,:\,\C^{n} \,\longrightarrow\, \C$ sending $\Lambda$ to a 
lattice $\Lambda'$ in $\C$. Then $\widetilde{\pi}$ descends to a holomorphic fibration $\pi\,:\, T 
\,\longrightarrow\, \C / \Lambda'$ over the elliptic curve $\C / \Lambda'$. Choose a non-constant 
meromorphic function $u$ on the elliptic curve $\C/ \Lambda'$ and consider the meromorphic 
closed one-form $\Omega\,=\,\pi^{*}(udz) + \omega$ on $T$, where $\omega$ is any (translation invariant) 
holomorphic one-form on $T$ and $dz$ is a nontrivial holomorphic section of the canonical bundle of 
$\C / \Lambda'$. It is easy to see that the foliation given by the kernel of $\Omega$ extend on 
all of $T$ as a nonsingular codimension one holomorphic foliation $\mathcal F$. This foliation 
is not invariant by all translations in $T$, more precisely, it is invariant only
by those translations that are spanned by vectors 
lying in the kernel of $\widetilde{\pi}$. The subgroup of translations preserving $\mathcal F$ is a 
subtorus $T'$ of complex codimension one in $T$ \cite{Gh}.

On the other hand, since the holomorphic tangent bundle of $T$ is trivial, we have a family 
of global (commuting) holomorphic vector fields $X_1,\, X_2,\, \ldots,\, X_n$ on $T$ which 
span the holomorphic tangent bundle $T_T$ at any point in $T$.

The holomorphic geometric structure $\phi \,=\, (\mathcal{F},\, X_1,\, \ldots,\, X_n)$, obtained by 
juxtaposing Ghys' foliation $\mathcal F$ and the vector fields $X_i$ is a holomorphic rigid 
geometric structure of non-affine type \cite{DG,Gr}. Local Killing fields of $\phi$ commute 
with all $X_i$, so they are linear combinations of $X_i$; so they extend as globally defined 
holomorphic vector fields on $T$. Since the local Killing fields of $\phi$ are translations 
which must also preserve $\mathcal F$, they span the subtorus $T'$ of $T$. In particular, 
the Killing algebra of $\phi$ has orbits of complex codimension one in $T$ and therefore 
$\phi$ is not locally homogeneous on any nontrivial open subset of $T$.

On the contrary, for holomorphic geometric structure of affine type, we think that the 
non-empty Zariski open set $\Om$ in Theorem \ref{ouvert dense} is all of the manifold. This 
was proved to be true in \cite{BD} for Moishezon manifolds (these are manifolds 
bimeromorphic to some complex projective manifold \cite{Mo}).

\subsection{Holomorphic Riemannian metric}

\begin{thm}\label{holo metric}
Let $Y$ be a compact complex manifold in Fujiki class $\mathcal C$ admitting a holomorphic
Riemannian metric $g$. Assume that there exists a cohomology class $\lbrack \alpha \rbrack 
\,\in\, H^{1,1}_{\partial \overline\partial}(Y,\,\R)$ that is nef and
has positive self-intersection. Then there is a finite unramified cover $\gamma\, :\, {\mathcal T}\, 
\longrightarrow\, Y$, where $\mathcal T$ is a complex torus, such that the pulled back holomorphic Riemannian 
metric $\gamma^*g$ on the torus ${\mathcal T}$ is translation invariant.
\end{thm}

\begin{proof}
Assume that $Y$ admits a holomorphic Riemannian metric $g$.
We have $T_Y\,=\, T^*_Y$, because $g$ gives a holomorphic isomorphism of $T^*_Y$ with $T_Y$.
This implies that the first Chern class of $T_Y$ vanishes. So Theorem \ref{Bochner principle}
holds for $Y$ because of Corollary \ref{cor1}. By Theorem \ref{Bochner principle}, there 
exists a non-empty Zariski open subset $\Om \, \subset\, Y$ endowed with a K\"ahler metric 
$\omega$, such that the restriction of the holomorphic tensor $g$ to $\Om$ is parallel with 
respect to the Levi-Civita connection on $T_{\Omega}$ for the K\"ahler metric $\omega$.

The following lemma proves that $g$ and $\omega$ are flat.

\begin{lemm}\label{lemc}
Let $U$ be an open subset of ${\C}^n$ in Euclidean topology, and let $\omega'$ be a K\"ahler metric on $U$. 
Suppose that there exists a holomorphic Riemannian metric $g'$ on $U$ such that the
tensor $g'$ is parallel with respect to the Levi-Civita connection for the
K\"ahler metric $\omega'$. Then the following three hold:
\begin{enumerate}
\item The K\"ahler metric $\omega'$ is flat.

\item The holomorphic Levi-Civita connection for $g'$ is flat.

\item The tensor $\omega'$ is flat with respect to the holomorphic Levi-Civita connection for $g'$.
\end{enumerate}
\end{lemm}

\begin{proof} Take any $u \,\in\, U$. Using de Rham's local splitting theorem, there exists a local 
decomposition of an open neighborhood $U^u\, \subset\, U$ of $u$ in ${\C}^n$ such that $(U^u,\, 
\omega')$ is a Riemannian product
\begin{equation}\label{e4}
(U^u,\, \omega')\,=\, (U_0,\, \omega_0) \times \ldots \times (U_p,\, \omega_p)\, ,
\end{equation}
where $(U_0, \,\omega_0)$ is a flat K\"ahler manifold and $(U_i,\, \omega_i)$
is an irreducible K\"ahler manifold for every $1\, \leq\, i\, \leq\, p$
(the reader is referred to \cite[Proposition 2.9]{GGK} 
for more details on this local K\"ahler decomposition).

For any $v\, \in\, U^u$, let $Q_v$ be the complex bilinear form associated to the quadratic form
$g'(v)$ on $T_v U^u$. Write $v\,=\, (v_0,\, v_1,\ldots, v_p)$ using \eqref{e4}. 
Since $g'$ is parallel with respect to $\omega'$, it follows that
for all $0\, \leq\, i,\, j\, \leq\, p$,
\begin{equation}\label{ev}
Q_v(w_i,\, w_j)\, =\, Q_v(h_i \cdot w_i,\, h_j \cdot w_j)
\end{equation}
for all $w_i \,\in\, T_{v_i} U_i$, $w_j \,\in\, T_{v_j} U_j$, for 
any $h_i\, \in\, \text{GL}(T_{v_i} U_i)$ in the holonomy group for $(U_i, \, \omega_i)$, and for any
$h_j\, \in\, \text{GL}(T_{v_j} U_j)$ in the holonomy
group for $(U_j, \, \omega_j)$. Assume that $i\, \not=\, j$. So at least one of $i$ and $j$
is different from zero. Assume that $j\, \not=\, 0$. From \eqref{ev} if follows
that
\begin{equation}\label{ev2}
Q_v(w_i,\, h_j \cdot w_j -h'_j \cdot w_j)\,=\, Q_v(w_i,\, w_j -w_j) \, =\,0
\end{equation}
for all $h_j,\, h'_j$ in the holonomy group for $(U_j,\, g_j)$
(set $h_i\,=\, \text{Id}$ in \eqref{ev}). Since $j\, >\, 0$, the holonomy 
group for $(U_j,\, g_j)$ is irreducible, which implies that the vector subspace of $T_{v_j} U_j$ 
generated by all elements of the form $h_j \cdot w_j -h'_j \cdot w_j$, where $w_j\, \in\, T_{v_j}U_j$ and $h_j,\, 
h'_j$ are elements of the holonomy group for $(U_j, \, \omega_j)$, is entire $T_{v_j} U_j$. Using 
this, from \eqref{ev2} it follows that $TU_i$ and $TU_j$ are $g'$--orthogonal for any $i \,\not=\, 
j$. Consequently, $g'$ is non-degenerate when restricted to all $(U_i,\, \omega_i)$, $0\, \leq\, i\, \leq\, p$.

To prove the first statement of the lemma
it suffices to show that $(U^u,\, \omega')\,=\, (U_0,\, \omega_0)$, meaning $p\,=\, 0$ in \eqref{e4}.

To prove $p\,=\, 0$ by contradiction, assume that $\omega'$ admits an (irreducible) factor 
$(U_1,\, \omega_1)$. The parallel transport for the Levi-Civita
connection for $\omega_1$ must preserve the 
restriction $g_1$ of $g'$ to $TU_1$ and also preserve the real part $h_1$ of $g_1$.
The real part $h_1$ of $g_1$ is a 
pseudo-Riemannian metric of signature $(n_1,\, n_1)$, where $n_1$ is the complex dimension of 
$U_1$. Consider the positive and the negative eigenspaces of
$h_1$ with respect to $\omega_1$. Since the parallel transport for the Levi-Civita
connection for $\omega_1$ preserves $g_1$,
the holonomy of $\omega_1$ preserves the positive and the negative eigenspaces of 
$h_1$ with respect to $\omega_1$. This is a contradiction, because the factor
$(U_1, \, \omega_1)$ is irreducible. This proves the first statement of the lemma.

To prove the second statement, since the K\"ahler metric $\omega'$ is flat, there exists local holomorphic
coordinates with respect to which $\omega'$ is the standard hermitian metric on ${\C}^n$. Take such a
holomorphic coordinate function on an open subset $U'\, \subset\, U$.
Therefore, on $U'$ parallel transports for the Levi-Civita 
connection for $\omega'$ are just translations in ${\C}^n$ in terms of this holomorphic coordinate function on
$U'$. Consequently, on $U'$ the holomorphic Riemannian metric 
$g'$ must be translation invariant (for the holomorphic coordinate function),
because $g'$ is invariant under the parallel transports for the 
Levi-Civita connection for $\omega'$. Hence the holomorphic Levi-Civita connection for $g'$ 
coincides with the standard affine connection on ${\C}^n$ in terms of the holomorphic coordinate function on 
$U'$. In particular, the holomorphic Levi-Civita connection for $g'$ is flat; this proves (2).

Since the holomorphic Levi-Civita connection for $g'$ coincides with the
standard affine connection on ${\C}^n$ in terms of a holomorphic coordinate function on $U'$
with respect to which $\omega'$ is the standard K\"ahler form on
${\C}^n$, it follows immediately that $\omega'$ is flat with respect to the holomorphic Levi-Civita
connection for $g'$. This completes the proof of the lemma.
\end{proof}

Continuing with the proof of Theorem \ref{holo metric}, let $\nabla^g$ be
the holomorphic Levi-Civita connection on $Y$ for the holomorphic Riemannian metric $g$. The
$C^\infty$ connection on $\Omega^{1,1}_Y\,=\, \Omega^{1,0}_Y\otimes \Omega^{0,1}_Y$ induced by
$\nabla^g$ is flat because $\nabla^g$ is flat by Lemma \ref{lemc}(2). Note that from Lemma \ref{lemc}(3)
it follows immediately that the section $\omega$ of
this flat bundle $\Omega^{1,1}_{\Om}$ is flat (covariant constant).

Since $\Om$ is a non-empty Zariski open subset of $Y$, the natural homomorphism
\begin{equation}\label{s1}
\pi_1(\Om, y_0) \, \longrightarrow\, \pi_1(Y,y_0)
\end{equation}
of fundamental groups is surjective, where $y_0\, \in\, \Om$. Using this
it can be deduced that the above flat section $\omega$ 
of $\Omega^{1,1}_{\Om}$ for the connection $\nabla^g$ extends to a flat section of $\Omega^{1,1}_Y$. 
Indeed, this follows immediately from the fact that the flat sections of a flat vector bundle are 
precisely the invariants of the monodromy representation. Note that for a $\pi_1(Y, y_0)$-module $V$, we 
have $V^{\pi_1(Y, y_0)}\, =\, V^{\pi_1(\Om, y_0)}$, because the homomorphism in \eqref{s1} is surjective. 
The flat section of $\Omega^{1,1}_Y$ (for the connection on it induced by $\nabla^g$) obtained by extending 
$\omega$ will be denoted by $\widehat{\omega}$.

Now consider the $\mathbb C$--linear homomorphism
$$
\widehat{\omega}'\, :\, T^{1,0}_Y\, \longrightarrow\, \Omega^{0,1}_Y
$$
given by $\widehat{\omega}$. It is connection preserving (for the connections
induced by $\nabla^g$), because the section $\omega$ is flat.
Since $\widehat{\omega}'$ is an isomorphism over $\Om$ (as $\omega$
is a K\"ahler form), and $\widehat{\omega}'$ is connection preserving,
it follows that $\widehat{\omega}'$ is an isomorphism over the entire $Y$.

Hence $\widehat{\omega}'$ defines a (nonsingular) hermitian structure on $Y$. This hermitian 
structure is K\"ahler because its restriction to $\Om$ is K\"ahler. This K\"ahler structure on $Y$ is 
flat because its restriction to $\Om$ is flat by Lemma \ref{lemc}(1). Therefore, $Y$ admits a finite 
unramified cover by a complex torus $\mathcal T$ such that the pull-back of $g$ to $\mathcal T$ is 
translation invariant \cite{Be}, \cite{Bo}.
\end{proof}


\end{document}